\documentclass[11pt]{article}

\usepackage{amsthm}
\usepackage{amssymb}
\usepackage{mathtools}
\usepackage{mathrsfs}
\usepackage{enumitem}
\usepackage[margin=1in]{geometry}

\usepackage{titlesec}
\usepackage{chngcntr}
\usepackage{subcaption}
\usepackage{xcolor}

\titleformat{\section}
{\centering\large\scshape\bfseries}{\thesection.}{.5em}{}

\newtheorem{thm}{Theorem}
\newtheorem{lem}[thm]{Lemma}

\theoremstyle{definition}

\theoremstyle{remark}

\newcommand{\pth}[1]{\left( #1 \right)}

\newcommand{\hor}[1]{\left[ #1 \right)}

\newcommand{\set}[1]{\left\{ #1 \right\}}
\newcommand{\abs}[1]{\left| #1 \right|}

\newcommand{\floor}[1]{\left\lfloor #1 \right\rfloor}

\newcommand{\sumabs}[1]{\bigg| #1 \bigg|}
\newcommand{\sumpth}[1]{\bigg( #1 \bigg)}
\newcommand{\fracp}[2]{\pth{\frac{#1}{#2}}}

\newcommand{\mand}{\qquad\text{and}\qquad}

\newcommand{\ep}{\varepsilon}

\renewcommand{\phi}{\varphi}

\counterwithin*{equation}{section}

\newcommand{\cole}{\mathcal{E}}\newcommand{\colm}{\mathcal{M}}\newcommand{\colt}{\mathcal{T}}

\usepackage{tikz}
\usepackage{pgfplots}
\usetikzlibrary{calc}
\pgfplotsset{compat=1.7}

\usepackage{xhfill}

\begin{document}
\bibliographystyle{abbrv}

\begin{center}
	{\large\bf THE FRACTIONAL SUM OF SMALL ARITHMETIC FUNCTIONS}\\[1.5em]
	{\scshape Joshua Stucky}
\end{center}
\vspace{.5em}

\begin{abstract}
Motivated by recent results, we study sums of the form
\[
S_f(x) = \sum_{n\leq x} f\pth{\floor{\frac{x}{n}}},
\]
where $f$ is an arithmetic function and $\floor{\cdot}$ denotes the greatest integer function. We show how the error term in the asymptotic formula for $S_f(x)$ can be improved in some specific cases.
\end{abstract}


%\noindent\textbf{Declaration of Interest}: The author declares that he has no conflict of interest.

\section{Introduction and Statement of Results}
Recently, there has been some interest in the sums
\[
S_f(x) = \sum_{n\leq x} f\pth{\floor{\frac{x}{n}}},
\]
where $f$ is an arithmetic function and $\floor{\cdot}$ denotes the greatest integer function. When $f(n) = n$, this is the classic Dirichlet divisor problem. Such sums do not seem to have a name in the literature yet, so we propose to call $S_f$ the ``fractional sum of $f$.'' In the present paper, we show how the error term in the asymptotic formula for $S_f(x)$ can be improved for some specific functions $f$ which are in some sense ``small.''

To motivate our results, we discuss some of the literature on the subject. If $f$ satisfies
\begin{equation}\label{eq:hyperbolaRequirement}
f(n) \ll n^\alpha(\log n)^\theta
\end{equation}
for some fixed $\alpha,\theta \geq 0$ with $\alpha < 1$, then Wu \cite{Wu2019} and Zhai \cite{Zhai2020} have independently used the hyperbola method to show that
\begin{equation}\label{eq:hyperbola}
S_f(x) = C_f x + O\pth{x^{\frac{1+\alpha}{2}}(\log x)^\theta},
\end{equation}
where
\[
C_f = \sum_{n=1}^\infty \frac{f(n)}{n(n+1)}.
\]
In particular, when $\abs{f(n)} \ll n^\ep$, we obtain
\begin{equation}\label{eq:trivialHalf}
	S_f(x) = C_f x  + O\pth{x^{\frac{1}{2}+\ep}}.
\end{equation}
The question of improving the error term beyond $\frac{1}{2}$ has been the subject of several papers (see \cite{Bordelles2020}, \cite{LiuWuYang2021}, \cite{MaSun2020}, and \cite{MaWu2020}, for instance). In particular, for $\tau$ the usual divisor-counting function and $\Lambda$ the von-Mangoldt function, we have
\[
S_\tau(x) = C_\tau x + O\pth{x^{\frac{19}{40}+\ep}} \mand S_\Lambda(x) = C_\Lambda x + O\pth{x^{\frac{9}{19}+\ep}},
\]
due to Bordell\`{e}s \cite{Bordelles2020} and Liu, Wu, and Yang \cite{LiuWuYang2021}, respectively (note $\frac{19}{40} = .475$ and $\frac{9}{19} =.4736...$). By applying a theorem of Jutila, we improve Bordell\`{e}s' estimate and prove

\begin{thm}\label{thm:511}
We have
\[
S_\tau(x) = C_\tau x + O(x^{5/11+\ep}).
\]
Note $\frac{5}{11} = 0.4545...$.
\end{thm}

We prove Theorem \ref{thm:511} in Section \ref{sec:511}. In a different direction, one can also improve the error term in (\ref{eq:hyperbola}) when $f$ is ``close to one,'' in a suitable sense. In this direction, Wu \cite{Wu2019} has shown that for $f = \phi$ the Euler phi function, we have
\begin{equation}\label{eq:WuPhi}
S_\phi(x) = C_\phi x + O(x^{1/3}(\log x)).
\end{equation}
In Section \ref{sec:Wu}, we generalize Wu's result and prove

\begin{thm}\label{thm:GeneralWu}
Suppose $f(n) = \sum_{d\mid n} g(d)$ and that
\begin{equation}\label{eq:gClose1}
\sum_{d\leq x} \abs{g(d)} \ll x^\alpha (\log x)^\theta
\end{equation}
for some $\alpha\in[0,1)$ and $\theta \geq 0$. Then
\[
S_f(x) = C_fx + O\pth{x^{\frac{1+\alpha}{3-\alpha}}(\log x)^\theta},
\]
where the implied constant depends only on $\alpha$. If $\alpha = 0$, then $\theta$ should be replaced by $\max(1,\theta)$.
\end{thm}

\noindent\textbf{Notation.} The symbols $O,o,\ll,\gg$ have their usual meanings. We use $n\asymp N$ to denote the condition $N < n \leq 2N$. The variable $\ep$ always denotes a positive arbitrarily small fixed real number, and $\delta$ is always 0 or 1. We also write $\psi(x) = x - \floor{x} - \frac{1}{2}$.

\section{A General Decomposition}

We begin with a general lemma that is useful in estimating the sums $S_f(x)$. The decomposition (\ref{eq:MainDecomp}) has appeared in more or less the same form in other works studying these sums (see  \cite{Wu2019} and \cite{Zhai2020}, for instance)

\begin{lem}\label{lem:decomp}
	Let $A,B\in[1,x^{1/2})$ be parameters to be chosen. Then
	\begin{equation}\label{eq:MainDecomp}
		S_f(x) = \colm(B) + O\Big(\abs{\colt_0(A,B)} + \abs{\colt_1(A,B)} + \cole_1(A) + \cole_2(B)\Big),
	\end{equation}
	where
	\[
	\begin{aligned}
		\colm(B) &= x \sum_{n \leq x/B} \frac{f(n)}{n(n+1)}, &&\cole_1(A) = \sum_{n \leq A} \abs{f(n)}, \\
		\colt_\delta(A,B) &= \sum_{A < n \leq x/B} f(n) \psi\fracp{x}{n+\delta}, \qquad &&\cole_2(B) = \sum_{n < B} \abs{f\pth{\floor{\frac{x}{n}}}}. 
	\end{aligned}
	\]
\end{lem}

\begin{proof}
	Since
	\[
	m = \floor{\frac{x}{n}} \qquad \iff \qquad \frac{x}{m+1} < n \leq \frac{x}{m},
	\]
	we can change variables and write
	\[
	\begin{aligned}
		S_f(x) &= \sum_{n < B} f\pth{\floor{\frac{x}{n}}} + \sum_{n \leq x/B} f(n) \pth{\floor{\frac{x}{n}} - \floor{\frac{x}{n+1}}} \\
		&= \sum_{n < B} f\pth{\floor{\frac{x}{n}}} + x\sum_{n \leq x/B} \frac{f(n)}{n(n+1)} + \sum_{n \leq x/B} f(n)\pth{\psi\pth{\frac{x}{n+1}} - \psi\pth{\frac{x}{n}}},
	\end{aligned}
	\]
	and the lemma follows.
	
\end{proof}

\section{Proof of Theorem \ref{thm:511}}\label{sec:511}

The proof of Theorem \ref{thm:511} is a combination of two ingredients. The first is 

\begin{lem}\label{lem:Bordelles}
	For $x$ sufficiently large, $f(n) \ll n^\ep$, and $N \in [x^{1/3},x^{1/2})$, we have for all $H \geq 1$
	\[
	\abs{S_f(x) - C_f x} \ll Nx^\ep + x^\ep \max_{\substack{N < D \leq x/N\\ \delta\in\set{0,1}}} \sumpth{\frac{D}{H} + \sum_{h\leq H}\frac{1}{h} \sumabs{\sum_{n\asymp D} f(n) e\fracp{hx}{n+\delta}}}.
	\]
\end{lem}

This lemma is originally due to Bordell\`{e}s \cite{Bordelles2020} and follows from Lemma \ref{lem:decomp} by inserting Vaaler's approximation for $\psi$ (see the apprendix of \cite{GrahamKolesnik} for a discussion of Vaaler's approximation).

The second ingredient is due to Jutila. The full statement of the result we need is somewhat lengthy, so we will state only the specific version we require. For a full statement and proof of this theorem, see Theorem 4.6 of \cite{Jutila}.\\

\begin{thm}\label{thm:Jutila}
	Let $2\leq M\leq M' \leq 2M$, and let
	\begin{equation}\label{eq:gConditions}
		g(z) = \frac{B}{z}\pth{1+O(F^{-1/3})}
	\end{equation}
	be a holomorphic function in the domain
	\[
	D = \set{z : \abs{z-x} < cM\ \text{for some}\ x\in[M,M']},
	\]
	where $c$ is some positive constant and
	\[
	F = \frac{\abs{B}}{M}.
	\]
	Suppose also that 
	\begin{equation}\label{eq:Fbound}
		M^{3/4} \ll F \ll M^{3/2}.
	\end{equation}
	Then
	\[
	\sumabs{\sum_{M\leq n\leq M'} \tau(n) e(g(n))} \ll M^{1/2}F^{1/3+\ep}.
	\]
\end{thm}

We are now ready to prove Theorem \ref{thm:511}. From Lemma \ref{lem:Bordelles}, we need to estimate the sums
\[
\sum_{n\asymp D} \tau(n) e\fracp{hx}{n+\delta}.
\]
In Theorem \ref{thm:Jutila}, we take
\[
M=D,\quad M'=2D,\quad g(n) = \frac{hx}{n+\delta}, \quad B = hx,\quad F = \frac{hx}{D}.
\]
Then the condition (\ref{eq:Fbound}) becomes
\[
\frac{D^{7/4}}{x} \ll h \ll \frac{D^{5/2}}{x},
\]
which is satisfied for all integers $h \in [1,H]$ so long as
\begin{equation}\label{eq:hConditions}
	D \ll x^{4/7-\ep} \mand H \ll \frac{D^{5/2}}{x}.
\end{equation}
Theorem \ref{thm:Jutila} then gives
\[
\sumabs{\sum_{n\asymp D} \tau(n) e\fracp{hx}{n+\delta}} \ll D^{1/6} h^{1/3}x^{1/3+\ep}.
\]
Summing over $h \leq H$, we have
\[
\abs{S_\tau(x) - C_\tau x } \ll N x^\ep +x^\ep\max_{N < D \leq x/N} \sumpth{\frac{D}{H} + D^{1/6} H^{1/3}x^{1/3+\ep}} \ll x^\ep \pth{N + \frac{x}{NH} + \frac{H^{1/3}}{N^{1/6}} x^{1/2}}.
\]
We complete the proof by choosing
\[
H = x^{3/8}N^{-5/8} \mand N = x^{5/11},
\]
which also ensures both conditions of (\ref{eq:hConditions}) are satisfied for all $D$ considered in the maximum, and that (\ref{eq:gConditions}) is satisfied.

\section{Proof of Theorem \ref{thm:GeneralWu}}\label{sec:Wu}

We follow Wu's \cite{Wu2019} method of proving (\ref{eq:WuPhi}). Let $A\in[1,x^{1/2})$ be a parameter to be chosen and use Lemma \ref{lem:decomp} with $B = A$. This gives
\[
S_f(x) = \colm + O\Big(\abs{\colt_0} + \abs{\colt_1} + \cole_1 + \cole_2\Big)
\]
with
\[
\begin{aligned}
\colm &= x \sum_{n \leq x/A} \frac{f(n)}{n(n+1)}, &&\cole_1 = \sum_{n \leq A} \abs{f(n)}, \\
\colt_\delta &= \sum_{A < n \leq x/A} f(n) \psi\fracp{x}{n+\delta}, \qquad &&\cole_2 = \sum_{n < A} \abs{f\pth{\floor{\frac{x}{n}}}}. 
\end{aligned}
\]
Note that (\ref{eq:gClose1}) trivially implies (\ref{eq:hyperbolaRequirement}) with the same $\alpha$ and $\theta$. Thus
\[
\colm = C_f x + O\sumpth{x\sum_{n > x/A} \frac{n^\alpha(\log n)^\theta}{n^2}} = C_f x + O(x^{\alpha}A^{1-\alpha}(\log x)^\theta),
\]
and likewise
\[
\cole_1 \ll A^{1+\alpha}(\log x)^\theta \mand \cole_2 \ll x^\alpha A^{1-\alpha} (\log x)^\theta.
\]
We have $A^{1+\alpha} < x^\alpha A^{1-\alpha}$ since $A < x^{1/2}$, and thus
\begin{equation}\label{eq:Sfdecomp}
S_f(x) = C_f x + O\pth{\abs{\colt_0} + \abs{\colt_1} + x^\alpha A^{1-\alpha}(\log x)^\theta}.
\end{equation}
We break $\colt_\delta$ (where $\delta$ is 0 or 1) into $O(\log x)$ sums of the form
\[
T(N) = \sum_{n\asymp N} f(n) \psi\fracp{x}{n+\delta},
\]
where $A\ll N \ll x/A$. Switching divisors, this is
\[
T(N) = \sum_{d\leq 2N} g(d) \sum_{n\asymp N/d} \psi\fracp{x}{dn+\delta}.
\]
We now employ Lemma 4.3 of \cite{GrahamKolesnik}. For any exponent pair $(k,l)$, we have
\begin{equation}\label{eq:TNexponentPair}
T(N) \ll \sum_{d\leq 2N} \abs{g(d)} \sumpth{x^{k/(k+1)} N^{(l-k)/(k+1)} d^{-l/(k+1)} + N^2x^{-1}d^{-1}}.
\end{equation}
The goal now is to choose an exponent pair $(k,l)$ such that, after applying partial summation and (\ref{eq:gClose1}), the first term on the right is dominated by the other error terms. For this, we need a sequence of specific exponent pairs given by

\begin{lem}\label{lem:ExponentPairs}
For any integer $n \geq 0$,
\[
(k_n,l_n) = \pth{\frac{1}{2^{n+2}-2}, \frac{2^{n+2}-n-3}{2^{n+2}-2}}
\]
is an exponent pair.
\end{lem}

\begin{proof}
Repeatedly apply the $A$ process to the pair $(\frac{1}{2},\frac{1}{2})$.

\end{proof}

Note that
\[
\frac{l_n}{k_n+1} = \frac{2^{n+2}-n-3}{2^{n+2}-1} = 1 - \frac{n+2}{2^{n+2}-1},
\]
and so $l_n/(k_n+1)$ is strictly increasing as $n\to\infty$. We define $(k_{-1},l_{-1}) = (1,0)$ (which is not an exponent pair) and divide the interval $[0,1)$ into subintervals
\[
I_n = \hor{\frac{l_{n-1}}{k_{n-1}+1}, \frac{l_n}{k_n+1}}, \qquad n\geq 0.
\]
Suppose that $\alpha \in I_n$. Partial summation and (\ref{eq:gClose1}) give
\begin{equation}\label{eq:PartialSummation}
\sum_{d\leq 2N} \frac{\abs{g(d)}}{d^{l_n/(k_n+1)}} \ll 1 \mand \sum_{d\leq 2N} \frac{\abs{g(d)}}{d} \ll 1
\end{equation}
since $\alpha < l_n/(k_n+1) < 1$. From (\ref{eq:TNexponentPair}) and (\ref{eq:PartialSummation}), we have
\[
T(N) \ll x^{k_n/(k_n+1)} N^{(l_n-k_n)/(k_n+1)} + N^2x^{-1}.
\]
Executing the dyadic sum in $N$ gives
\[
\begin{aligned}
S_f(x) &= C_f x + O\pth{x^{l_n/(k_n+1)}A^{(k_n-l_n)/(k_n+1)}(\log x)^{\delta(n=0)} + xA^{-2} + x^{\alpha} A^{1-\alpha} (\log x)^\theta} \\
&= C_f x + O\pth{E_1(\log x)^{\delta(n=0)}+E_2+E_3},
\end{aligned}
\]
where $\delta(n=0)$ is 1 if $n=0$ and $0$ otherwise. If $\alpha = 0$, then $l_n=k_n = \frac{1}{2}$ and we have
\[
S_f(x) = C_f x + O\pth{x^{1/3}(\log x) + xA^{-2} + A (\log x)^\theta}.
\]
Choosing $A = x^{1/3}$ gives the desired result with $\max(1,\theta)$ in place of $\theta$. Suppose now that $\alpha > 0$. We will show that $E_1$ is always dominated by $E_2$ or $E_3$.  We have $E_1 \leq E_2$ if $A \leq x^{B_1(n)}$, where
\[
B_1(n)= \frac{k_n-l_n+1}{3k_n-l_n+2} = \frac{n+2}{n+2^{n+2}+2},
\]
and likewise $E_1 \leq E_3$ if $A \geq x^{B_2(n)}$, where
\[
B_2(n) = \frac{\frac{l_n}{k_n+1}-\alpha}{\frac{l_n}{k_n+1}-\alpha+\frac{1}{k_n+1}} \leq \frac{\frac{l_n}{k_n+1}-\frac{l_{n-1}}{k_{n-1}+1}}{\frac{l_n}{k_n+1}-\frac{l_{n-1}}{k_{n-1}+1}+\frac{1}{k_n+1}} = \frac{n+2^{-n-1}}{ n+2^{n+2} + 3(2^{-n-1}) -4}.
\]
The inequality $B_2(n) \leq B_1(n)$ for $n=0,1,2$ may be verified by direct computation. For $n \geq 3$, we have
\[
\begin{aligned}
\frac{B_1(n)}{B_2(n)} \geq \fracp{n+2}{n+2^{-n-1}} \fracp{n+2^{n+2}+3(2^{-n-1}) -4}{n+2^{n+2}+2} \geq \pth{1+\frac{1}{n}}\pth{1-\frac{1}{2^{n-1}}} \geq 1
\end{aligned}
\]
We thus have $E_1 \leq \max(E_2,E_3)$. In the case $n=0$, since $\alpha\neq 0$, we actually have $E_1 \leq x^{-\lambda}\max(E_2,E_3)$ for some sufficiently small positive $\lambda$, and thus we may ignore the factor $(\log x)^{\delta(n=0)}$. For $\alpha > 0$, we thus have
\[
S_f(x) = C_f x + O\pth{(xA^{-2} + x^{\alpha} A^{1-\alpha})(\log x)^\theta}.
\]
Choosing $A = x^{(1-\alpha)/(3-\alpha)}$ completes the proof.

\section{Remarks on Theorem \ref{thm:GeneralWu}}

Theorem \ref{thm:GeneralWu} is generally good when $g$ is such that either $\abs{g}$ is small or $g$ is supported on a sparse set. For instance, let $\beta \in (0,1]$ and let $\sigma_\beta$ denote the sum of $\beta$th powers of divisors, $\sigma_\beta(n) = \sum_{d\mid n} d^\beta$. For $f(n) = \sigma_\beta(n) n^{-\beta}$, we have $f(n) \ll n^\ep$, and thus (\ref{eq:trivialHalf}) holds. However, we also have (\ref{eq:gClose1}) with $\alpha = 1-\beta$ and $\theta = 0$ (unless $\beta = 1$, in which case $\theta=1$). Thus
\[
S_f(x) = C_f x + O\pth{x^\frac{2-\beta}{2+\beta}(\log x)^{\delta(\beta=0)}},
\]
This is superior to (\ref{eq:hyperbola}) so long as $\beta \geq \frac{2}{3}$.

On the other hand, it is possible for Theorem \ref{thm:GeneralWu} to give a worse result than (\ref{eq:hyperbola}). For instance, let $f$ be the indicator function for the squarefree numbers so that $f(n) = \sum_{d\mid n} g(d)$ with
\[
g(d) = \begin{cases}
\mu(l) & \text{if $d=l^2$}, \\
0 & \text{otherwise},
\end{cases}
\]
We have
\[
\sum_{d\leq x} \abs{g(d)} = \sum_{l\leq \sqrt{x}} \mu^2(l) \asymp \sqrt{x},
\]
and so Theorem \ref{thm:GeneralWu} only gives
\[
S_f(x) = C_f x + O\pth{x^{\frac{3}{5}}}.
\]
More generally, if $f$ is the indicator function of the $k$-free numbers, then the same argument yields
\[
S_f(x) = C_f x + O\pth{x^{\pth{1+\frac{1}{k}}\pth{3-\frac{1}{k}}^{-1}}},
\]
which is superior to (\ref{eq:hyperbola}) so long as $k > 3$. 
\bibliography{references}
\end{document}